\documentclass[11pt]{article}
\usepackage[a4paper,margin=1in]{geometry}
\usepackage[T1]{fontenc}
\usepackage[utf8]{inputenc}
\usepackage{lmodern}
\usepackage{amsmath,amsthm,amssymb,mathtools}
\usepackage[shortlabels]{enumitem}
\usepackage{graphicx}
\usepackage{booktabs}
\usepackage{microtype}
\usepackage{tikz}
\usepackage{lineno}
\usepackage[numbers,sort&compress]{natbib}
\usepackage[colorlinks=true,linkcolor=blue,citecolor=blue,urlcolor=blue]{hyperref}
\hypersetup{
pdftitle={Forbidden graded posets in middle levels: supersaturation, enumeration and randomness},
pdfauthor={Jiangdong Ai, Mingyu Liu, Zhiyu Song},
pdfkeywords={Boolean lattice, forbidden subposets, supersaturation, hypergraph containers, random Turan problems}
}

\theoremstyle{plain}
\newtheorem{theorem}{Theorem}[section]
\newtheorem{lemma}[theorem]{Lemma}
\newtheorem{proposition}[theorem]{Proposition}
\newtheorem{corollary}[theorem]{Corollary}

\title{Forbidden graded posets in middle levels:\\
supersaturation, enumeration and randomness}
\author{
Jiangdong Ai\thanks{School of Mathematical Sciences and LPMC, Nankai University, Tianjin 300071, China. Email: \texttt{jd@nankai.edu.cn}. Funded by the National Natural Science Foundation of China (Nos. 12522117), the Fundamental and Interdisciplinary Disciplines Breakthrough Plan of the Ministry of Education of China (JYB2025XDXM207), and the Fundamental Research Funds for the Central Universities, Nankai University.}
\and Mingyu Liu\thanks{School of Mathematical Sciences and LPMC, Nankai University, Tianjin 300071, China. Email: \texttt{1120260046@mail.nankai.edu.cn}.}
\and Zhiyu Song\thanks{School of Mathematical Sciences and LPMC, Nankai University, Tianjin 300071, China. Email: \texttt{2210655@mail.nankai.edu.cn}.}
}

\date{}

\begin{document}
\maketitle

\begin{abstract}
Let $P$ be a fixed finite poset whose maximal chains all have $h$ elements, and let $\mathcal W_{n,h}$ be the union of $h$ consecutive middle levels of $2^{[n]}$. We prove that the maximum size of a $P$-free subfamily of $\mathcal W_{n,h}$ is $(\pi(P)+o(1))M_n$ for a constant $\pi(P)$, where $M_n=\binom n{\lfloor n/2\rfloor}$, and that the number of $P$-free subfamilies is $2^{(\pi(P)+o(1))M_n}$. For connected $P$ we prove supersaturation above $(\pi(P)+\epsilon)M_n$ and determine, in terms of the coordinate patterns of the copies of $P$ in the window, a range of $p$ in which the largest $P$-free subfamily of the random family $\mathcal W_{n,h}(p)$ has size $(\pi(P)+o(1))pM_n$ and a range in which it has size $(h+o(1))pM_n$; the two ranges meet whenever all copies have the same pattern. For the Boolean poset $Q_d$, which has height $d+1$, the threshold order in the corresponding window $\mathcal W_{n,d+1}$ is $n^{-d/(2^d-1)}$, and $\pi_d=\pi(Q_d)$ satisfies $(d+1)\lambda_d\le\pi_d\le d+1-\binom d{\lfloor d/2\rfloor}^{-1}$, where $\lambda_d$ is the vertex Tur\'an density of $d$-dimensional subcubes in the hypercube. Together with a construction of Ellis, Ivan and Leader this gives $d+1-\pi_d=\exp(-\Theta(d))$. For the diamond, known bounds give $2.1479\le\pi_2\le2.15121$.
\end{abstract}


\section{Introduction}

Let $2^{[n]}$ be the Boolean lattice, that is, the family of all subsets of $[n]=\{1,\dots,n\}$ ordered by inclusion, and let $M_n=\binom{n}{\lfloor n/2\rfloor}$. For a finite poset $P$, a family $\mathcal F\subseteq2^{[n]}$ contains a weak copy of $P$ if there is an injective order-preserving map from $P$ to $\mathcal F$, and $\mathcal F$ is $P$-free otherwise. Throughout, copies are counted as subfamilies rather than as labelled embeddings. Let $\operatorname{La}(n,P)$ denote the maximum size of a $P$-free subfamily of $2^{[n]}$. For $d\ge1$, let $Q_d=(2^{[d]},\subseteq)$ be the $d$-dimensional Boolean poset. The case $d=2$ is the diamond, that is, four distinct sets $A,B,C,D$ with $A\subset B\subset D$ and $A\subset C\subset D$.

The forbidden subposet problem asks for the asymptotics of $\operatorname{La}(n,P)$. Bukh~\cite{Bukh2009Set} determined $\operatorname{La}(n,P)$ asymptotically for every poset $P$ whose Hasse diagram is a tree, and supersaturation, enumeration and random versions for tree posets were obtained in \cite{Balogh2025Tree,Jiang2025Tree}. The diamond is the smallest poset whose Hasse diagram contains a cycle, and its extremal problem is still open. Gr\'osz, Methuku and Tompkins~\cite{Grosz2018Upper} proved that
\[
\operatorname{La}(n,Q_2)\le\left(\frac{3+\sqrt2}{2}+o(1)\right)M_n<(2.20711+o(1))M_n.
\]
If the family is restricted to three consecutive middle levels, Manske and Shen~\cite{Manske2013Three} proved the upper bound $((3+2\sqrt3)/3+o(1))M_n$, and Balogh, Hu, Lidick\'y and Liu~\cite{Balogh2014Upper} improved it to $(2.15121+o(1))M_n$. Recently, Tompkins~\cite{Tompkins2026Improved} constructed a diamond-free family in the same three levels of size
\[
\left(2+\frac{29}{135}\prod_{i=1}^{\infty}(1-4^{-i})+o(1)\right)M_n=(2.147908\dots+o(1))M_n,
\]
which disproves the conjecture that $\operatorname{La}(n,Q_2)=(2+o(1))M_n$. More generally, Bukh~\cite{Bukh2009Set} and Griggs and Lu~\cite{Griggs2009Families} conjectured that $\operatorname{La}(n,P)=(e(P)+o(1))M_n$ for every fixed poset $P$, where $e(P)$ is the largest number of consecutive levels whose union is always $P$-free. Ellis, Ivan and Leader~\cite{Ellis2024Daisies} disproved this prediction for $Q_d$ with $d\ge4$ by a construction inside $d+1$ consecutive middle levels; Tompkins~\cite{Tompkins2026Improved} and Patk\'os~\cite{Patkos2026B3} subsequently obtained counterexamples for $d=2$ and $d=3$, respectively.

The problem is also related to the vertex Tur\'an problem in the hypercube. A copy of $Q_d$ inside $d+1$ consecutive levels of $2^{[n]}$ is the vertex set of a $d$-dimensional subcube of the hypercube $\mathcal Q_n$, and one asks for the maximum density of a subset of $V(\mathcal Q_n)$ containing no such subcube. Kostochka~\cite{Kostochka1976Piercing} and, independently, Johnson and Entringer~\cite{Johnson1989Largest} showed that this density is $2/3$ for $d=2$. Alon, Krech and Szab\'o~\cite{Alon2007Turan} and Johnson and Talbot~\cite{Johnson2010Vertex} studied the general case, and Johnson and Talbot asked whether the maximum density is always attained by a union of levels. Ellis, Ivan and Leader~\cite{Ellis2024Daisies,Ellis2025Small} answered this question in the negative for all large $d$. The Boolean case considered here restricts this problem to the $d+1$ levels that a single subcube can occupy. A random translation of the hypercube relates the two densities: we show that the window density is at least $d+1$ times the vertex Tur\'an density of the hypercube. In particular, for large $d$ a $Q_d$-free family can fill the window up to an exponentially small deficit.

Supersaturation, enumeration and random versions of forbidden subposet problems were studied by Gerbner, Nagy, Patk\'os and Vizer~\cite{Gerbner2021Supersaturation}. They proposed three strengthenings of the conjecture above, all stated in terms of $e(P)$, the largest number of consecutive middle levels whose union is $P$-free: every family of size $(e(P)+\epsilon)M_n$ contains many copies of $P$, the number of $P$-free families is $2^{(e(P)+o(1))M_n}$, and the largest $P$-free subfamily of the random family has size $(e(P)+o(1))pM_n$ above a certain threshold. Note that $e(Q_d)=d$, so the constructions in \cite{Ellis2024Daisies,Tompkins2026Improved,Patkos2026B3} show that all three statements fail for $Q_d$ with $d\ge2$. Jiang, Longbrake and Yepremyan~\cite{Jiang2026Number} proved that for every poset $P$ that is not an antichain, the number of induced $P$-free families in $2^{[n]}$ is $2^{O(\operatorname{La}^{*}(n,P))}$, where $\operatorname{La}^{*}(n,P)$ is the induced extremal number. Their proof combines a balanced supersaturation theorem, obtained by averaging over random interval sublattices of $2^{[n]}$, with the container lemma of Balogh and Samotij~\cite{Balogh2020Efficient}. The codegrees in their supersaturation theorem are controlled by a parameter $\mu(P)$, defined through the dimension of the interval spanned by an induced copy of a subposet of $P$, and $\mu(Q_d)=d/(2^d-1)$. They also note that the exponent $O(\operatorname{La}^{*}(n,P))$ would be best possible in the form $(1+o(1))\operatorname{La}^{*}(n,P)$ whenever $\operatorname{La}^{*}(n,P)/M_n$ converges. In Section~\ref{sec:examples} we compare $\mu(P)$ with the parameters arising from copies in a fixed window; the two need not agree.

We call a finite nonempty poset $P$ graded of height $h$ if all its maximal chains have $h$ elements; its ranks are $0,\dots,h-1$. Put
\[
k_{n,h}=\left\lfloor\frac{n-h+1}{2}\right\rfloor,\qquad
\mathcal W_{n,h}=\bigcup_{i=0}^{h-1}\binom{[n]}{k_{n,h}+i}.
\]
Write $\operatorname{ex}_P(n)$ and $\operatorname{forb}_P(n)$ for the largest size and the number of $P$-free subfamilies of $\mathcal W_{n,h}$. Since exactly $h$ levels are used, every copy of $P$ in $\mathcal W_{n,h}$ preserves ranks, and for connected $P$ this leaves only finitely many coordinate patterns for a copy.

\begin{theorem}\label{thm:general-main}
For every fixed nonempty graded poset $P$,
\[
\pi(P)=\lim_{n\to\infty}\frac{\operatorname{ex}_P(n)}{M_n}
\quad\text{exists},\qquad
\operatorname{forb}_P(n)=2^{(\pi(P)+o(1))M_n}.
\]
If the Hasse diagram of $P$ has connected components $P_1,\ldots,P_c$, then $\pi(P)=\max_i\pi(P_i)$.
\end{theorem}

For connected $P$ of height $h\ge2$ that has a copy in $\mathcal W_{n,h}$, Section~\ref{sec:general} defines a finite set $\mathfrak T(P)$ of types, one for each coordinate pattern of a copy. A type $\mathcal T$ is a family of $|P|$ subsets of a finite set, and we put
\[
\nu_{\mathcal T}=\min_{\substack{\mathcal J\subseteq\mathcal T\\|\mathcal J|\ge2}}
\frac{|\bigcup\mathcal J\setminus\bigcap\mathcal J|}{|\mathcal J|-1},\qquad
\nu(P)=\min_{\mathcal T\in\mathfrak T(P)}\nu_{\mathcal T},\qquad
\nu_+(P)=\max_{\mathcal T\in\mathfrak T(P)}\nu_{\mathcal T}.
\]
These are positive rationals determined by a finite computation. We prove supersaturation in which every type is weighted by its own number of copies, and we apply the container theorem to the types separately. This gives a largest $P$-free subfamily of $\mathcal W_{n,h}(p)$ of size $(\pi(P)+o(1))pM_n$ with high probability when $pn^{\nu(P)}\to\infty$, and deleting suitable subfamilies of copies gives size $(h+o(1))pM_n$ when $pn^{\nu_+(P)}\to0$ and $pM_n\to\infty$. Hence the threshold order is $n^{-\nu(P)}$ whenever $\nu(P)=\nu_+(P)$, in particular whenever $P$ has a single type. Otherwise these are two sufficient conditions, and we do not claim anything about the range in between. Posets with no copy in their height window are trivial: the whole window is then $P$-free.

We first present the Boolean case, where there is a single type and all parameters have simple expressions. For $d\ge1$ and $n\ge d$, the corresponding height window is given by
\[
k_{n,d+1}=\left\lfloor\frac{n-d}{2}\right\rfloor,\qquad
\mathcal W_{n,d+1}=\bigcup_{i=0}^{d}\binom{[n]}{k_{n,d+1}+i}.
\]
Thus $\mathcal W_{n,d+1}$ consists of $d+1$ consecutive levels centred as closely as possible around $n/2$; for $d=2$ it is $\binom{[n]}{\lfloor n/2\rfloor-1}\cup\binom{[n]}{\lfloor n/2\rfloor}\cup\binom{[n]}{\lfloor n/2\rfloor+1}$. Let
\[
\operatorname{ex}_d(n)=\max\{|\mathcal F|:\mathcal F\subseteq\mathcal W_{n,d+1}\text{ is }Q_d\text{-free}\},
\]
and let $\operatorname{forb}_d(n)$ be the number of $Q_d$-free subfamilies of $\mathcal W_{n,d+1}$. We show that the extremal density of this problem, rather than the integer $e(Q_d)=d$, is the constant that governs supersaturation, enumeration and the random version inside $\mathcal W_{n,d+1}$. Our first Boolean result shows that the extremal density exists and gives a supersaturation statement.

\begin{theorem}\label{thm:1.1}
For every fixed $d\ge1$, the limit
\[
\pi_d=\lim_{n\to\infty}\frac{\operatorname{ex}_d(n)}{M_n}
\]
exists. Moreover, for every $\epsilon>0$ there are $c=c(d,\epsilon)>0$ and $n_0=n_0(d,\epsilon)$ such that for every $n\ge n_0$, every $\mathcal F\subseteq\mathcal W_{n,d+1}$ with $|\mathcal F|\ge(\pi_d+\epsilon)M_n$ contains at least $cn^dM_n$ copies of $Q_d$.
\end{theorem}

We shall see that every weak copy of $Q_d$ in $\mathcal W_{n,d+1}$ is a subcube of the form $\{A\cup Y:Y\subseteq X\}$ with $|A|=k_{n,d+1}$ and $|X|=d$. The supersaturation part of Theorem~\ref{thm:1.1} is proved by averaging over lower-dimensional sections of $\mathcal W_{n,d+1}$. Our second result is a counting theorem.

\begin{theorem}\label{thm:1.2}
For every fixed $d\ge1$,
\[
\operatorname{forb}_d(n)=2^{(\pi_d+o(1))M_n}.
\]
\end{theorem}

The lower bound follows from an extremal family and its subfamilies. For the upper bound we view the subcubes as the edges of a $2^d$-uniform hypergraph on $\mathcal W_{n,d+1}$, compute its codegrees and apply the hypergraph container method \cite{Balogh2015Independent,Saxton2012HypergraphC}. The same codegree computation gives the threshold in the random version. Let $\mathcal W_{n,d+1}(p)$ be the random subfamily of $\mathcal W_{n,d+1}$ in which every member is retained independently with probability $p$, and let $\alpha_d(\mathcal W_{n,d+1}(p))$ be the maximum size of a $Q_d$-free subfamily of $\mathcal W_{n,d+1}(p)$.

\begin{theorem}\label{thm:1.3}
Fix $d\ge1$ and let $\gamma_d=d/(2^d-1)$.
\begin{enumerate}
\item[\textup{(i)}] If $pn^{\gamma_d}\to\infty$, then with high probability
$\alpha_d(\mathcal W_{n,d+1}(p))=(\pi_d+o(1))pM_n$.
\item[\textup{(ii)}] If $pn^{\gamma_d}\to0$ and $pM_n\to\infty$, then with high probability
$\alpha_d(\mathcal W_{n,d+1}(p))=(d+1+o(1))pM_n$.
\end{enumerate}
\end{theorem}

Thus $n^{-d/(2^d-1)}$ is the order of the threshold at which the subcubes start to reduce the largest $Q_d$-free subfamily by a constant factor. For the diamond the exponent is $2/3$, and Theorems~\ref{thm:1.1}--\ref{thm:1.3} together with the bounds of \cite{Balogh2014Upper,Tompkins2026Improved} give the following.

\begin{corollary}\label{cor:diamond}
Let $\pi_{\diamond}=\pi_2$. Then
\[
2+\frac{29}{135}\prod_{i=1}^{\infty}(1-4^{-i})\le\pi_{\diamond}\le2.15121.
\]
Moreover, $\operatorname{forb}_2(n)=2^{(\pi_{\diamond}+o(1))M_n}$, and if $p\gg n^{-2/3}$, then with high probability the largest diamond-free subfamily of $\mathcal W_{n,3}(p)$ has size $(\pi_{\diamond}+o(1))pM_n$.
\end{corollary}

Numerically, $2.1479080635\dots\le\pi_{\diamond}\le2.15121$. Note that the proofs do not use the value of $\pi_d$ or any structural information about extremal families, so any improvement of the bounds on $\pi_d$ transfers directly to Theorems~\ref{thm:1.2} and~\ref{thm:1.3}. Part (ii) of Theorem~\ref{thm:1.3} follows from the deletion argument of Gerbner, Nagy, Patk\'os and Vizer~\cite[Proposition~5]{Gerbner2021Supersaturation}, whose parameters satisfy $e(Q_d)=d$ and $x(Q_d)=d$; we include the short proof for completeness. The new part is (i), the dense regime inside $\mathcal W_{n,d+1}$ with the constant $\pi_d$.

The rest of the paper is organised as follows. Section~\ref{sec2} introduces canonical subcubes and centred sections, proves the existence of $\pi_d$, and relates it to the full-cube vertex Tur\'an density. Section~\ref{sec4} proves Boolean supersaturation and the codegree estimates. Sections~\ref{sec6} and~\ref{sec8} establish the Boolean enumeration and random results. Section~\ref{sec:general} treats connected graded posets, and Section~\ref{sec:examples} treats disconnected posets, computes the parameters for several examples and compares them with the parameter $\mu(P)$ of \cite{Jiang2026Number}.

\section{Canonical subcubes and the extremal density}\label{sec2}

For disjoint sets $A,X\subseteq[n]$, let
\[
\mathcal Q(A,X)=\{A\cup Y:Y\subseteq X\}.
\]
If $|X|=d$, then $\mathcal Q(A,X)$ is a copy of $Q_d$. We call $\mathcal Q(A,X)$ a canonical subcube of $\mathcal W_{n,d+1}$ if $|A|=k_{n,d+1}$ and $|X|=d$. For $d=2$, the canonical subcubes are exactly the rank-two intervals $\{A,A\cup\{x\},A\cup\{y\},A\cup\{x,y\}\}$ in the three middle levels.

\begin{lemma}\label{lem:2.1}
Every weak copy of $Q_d$ in $\mathcal W_{n,d+1}$ is a canonical subcube. In particular, weak and induced copies of $Q_d$ in $\mathcal W_{n,d+1}$ coincide.
\end{lemma}

\begin{proof}
Let $\varphi:Q_d\to\mathcal W_{n,d+1}$ be an injective order-preserving map. A maximal chain of $Q_d$ has $d+1$ elements and $\mathcal W_{n,d+1}$ has exactly $d+1$ levels, so the image of every maximal chain contains one set from each level. Since every $S\subseteq[d]$ lies in a maximal chain, we have $|\varphi(S)|=k_{n,d+1}+|S|$. Let $A=\varphi(\emptyset)$. For each $i\in[d]$ there is $x_i\notin A$ with $\varphi(\{i\})=A\cup\{x_i\}$, and $x_1,\dots,x_d$ are distinct since $\varphi$ is injective. For $S\subseteq[d]$, order preservation gives $A\cup\{x_i:i\in S\}\subseteq\varphi(S)$, and both sides have size $k_{n,d+1}+|S|$. Hence $\varphi(S)=A\cup\{x_i:i\in S\}$ and $\varphi(Q_d)=\mathcal Q(A,\{x_1,\dots,x_d\})$. Clearly, the order induced on a canonical subcube is $Q_d$, which proves the last assertion.
\end{proof}

Let $r=2^d$ and let $\mathcal H_{n,d}$ be the $r$-uniform hypergraph with vertex set $\mathcal W_{n,d+1}$ whose edges are the canonical subcubes. By Lemma~\ref{lem:2.1}, the independent sets of $\mathcal H_{n,d}$ are exactly the $Q_d$-free subfamilies of $\mathcal W_{n,d+1}$, so
\[
\alpha(\mathcal H_{n,d})=\operatorname{ex}_d(n),\qquad i(\mathcal H_{n,d})=\operatorname{forb}_d(n),
\]
where $\alpha$ and $i$ denote the independence number and the number of independent sets.

\begin{lemma}\label{lem:2.2}
For every fixed $d\ge1$,
\[
v(\mathcal H_{n,d})=(d+1+o(1))M_n,\qquad
e(\mathcal H_{n,d})=\binom{n}{k_{n,d+1}}\binom{n-k_{n,d+1}}{d}=\left(\frac{1}{2^dd!}+o(1)\right)n^dM_n.
\]
\end{lemma}

\begin{proof}
Every level whose index differs from $n/2$ by $O_d(1)$ has size $(1+o(1))M_n$, which gives the first formula. A canonical subcube $\mathcal Q(A,X)$ is determined by $A$ and $X$, and there are $\binom{n}{k_{n,d+1}}$ choices for $A$ and $\binom{n-k_{n,d+1}}{d}$ choices for $X$. Since $\binom{n}{k_{n,d+1}}=(1+o(1))M_n$ and $n-k_{n,d+1}=n/2+O_d(1)$, the second formula follows.
\end{proof}

We now introduce the averaging device used in the paper. Fix integers $n\ge m\ge d$. A centred $m$-section of $\mathcal W_{n,d+1}$ is a pair $(B,Y)$ of disjoint subsets of $[n]$ with
\[
|Y|=m,\qquad |B|=k_{n,d+1}-k_{m,d+1}.
\]
Such pairs exist since $0\le k_{n,d+1}-k_{m,d+1}\le n-m$. The family associated with $(B,Y)$ is
\[
\mathcal W(B,Y)=\{B\cup Z:Z\subseteq Y,\ k_{m,d+1}\le|Z|\le k_{m,d+1}+d\}\subseteq\mathcal W_{n,d+1}.
\]
The map $Z\mapsto B\cup Z$ identifies $\mathcal W_{m,d+1}$ with $\mathcal W(B,Y)$, and identifies $\mathcal H_{m,d}$ with the subhypergraph of $\mathcal H_{n,d}$ induced by $\mathcal W(B,Y)$. Let $\mathfrak S_{n,m}$ be the set of all centred $m$-sections, and let
\[
\rho_{m,d+1}=\min_{0\le i\le d}\frac{\binom{m}{k_{m,d+1}+i}}{M_m}.
\]
Note that $\rho_{m,d+1}\le1$ and $\rho_{m,d+1}\to1$ as $m\to\infty$ for fixed $d$. Clearly, if $\mathcal F\subseteq\mathcal W_{n,d+1}$ is $Q_d$-free, then so is $\mathcal F\cap\mathcal W(B,Y)$ for every section $(B,Y)$.

\begin{lemma}\label{lem:3.1}
Let $\mathcal F\subseteq\mathcal W_{n,d+1}$ and let $(B,Y)$ be a uniformly random member of $\mathfrak S_{n,m}$. Then
\[
\mathbb E|\mathcal F\cap\mathcal W(B,Y)|
=\sum_{i=0}^{d}\left|\mathcal F\cap\binom{[n]}{k_{n,d+1}+i}\right|\frac{\binom{m}{k_{m,d+1}+i}}{\binom{n}{k_{n,d+1}+i}}
\ge\rho_{m,d+1}\frac{M_m}{M_n}|\mathcal F|.
\]
\end{lemma}

\begin{proof}
Every section contains exactly $\binom{m}{k_{m,d+1}+i}$ sets of level $k_{n,d+1}+i$. By symmetry, each set of this level lies in the same number of sections, so a fixed set of level $k_{n,d+1}+i$ lies in a random section with probability $\binom{m}{k_{m,d+1}+i}/\binom{n}{k_{n,d+1}+i}$. Summing the indicators over $\mathcal F$ gives the equality. The inequality follows from $\binom{m}{k_{m,d+1}+i}\ge\rho_{m,d+1}M_m$ and $\binom{n}{k_{n,d+1}+i}\le M_n$.
\end{proof}

\begin{proposition}\label{prop:3.2}
For every fixed $d\ge1$, the sequence $\operatorname{ex}_d(n)/M_n$ converges.
\end{proposition}

\begin{proof}
Let $a_n=\operatorname{ex}_d(n)/M_n$. Fix $n\ge m\ge d$ and let $\mathcal F\subseteq\mathcal W_{n,d+1}$ be a $Q_d$-free family of size $\operatorname{ex}_d(n)$. The intersection of $\mathcal F$ with every centred $m$-section is $Q_d$-free, so Lemma~\ref{lem:3.1} gives
\[
\rho_{m,d+1}\frac{M_m}{M_n}\operatorname{ex}_d(n)\le\mathbb E|\mathcal F\cap\mathcal W(B,Y)|\le\operatorname{ex}_d(m),
\]
that is, $a_n\le a_m/\rho_{m,d+1}$. Let $L=\liminf_{m\to\infty}a_m$ and choose $m_j\to\infty$ with $a_{m_j}\to L$. Taking the limit superior over $n$ and then letting $j\to\infty$, we obtain
\[
\limsup_{n\to\infty}a_n\le\lim_{j\to\infty}\frac{a_{m_j}}{\rho_{m_j,d+1}}=L.
\]
Hence $\limsup a_n\le\liminf a_n$ and the sequence converges.
\end{proof}

From now on, $\pi_d=\lim_{n\to\infty}\operatorname{ex}_d(n)/M_n$. The union of any $d$ of the $d+1$ levels of $\mathcal W_{n,d+1}$ is $Q_d$-free, so $\pi_d\ge d$. In fact $\pi_d>d$ for every $d\ge2$. For $d\ge4$, Ellis, Ivan and Leader~\cite{Ellis2024Daisies} kept $d$ of the $d+1$ levels and replaced the remaining middle level by a family of positive density that contains no copy of the middle level of $Q_d$, that is, no generalised daisy. For $d=2$ this is the construction of Tompkins~\cite{Tompkins2026Improved}, which gives $\pi_2\ge2.1479$, and for $d=3$ Patk\'os~\cite{Patkos2026B3} obtained $\pi_3\ge3.22$ by adapting it. His four layers agree with our window when $n$ is odd; taking complements when $n$ is even gives the stated bound in our convention.
On the other hand,
\[
\pi_d\le d+1-\binom{d}{\lfloor d/2\rfloor}^{-1}<d+1.
\]
Indeed, let $\mathcal F\subseteq\mathcal W_{n,d+1}$ be $Q_d$-free and let $\mathcal G=\mathcal W_{n,d+1}\setminus\mathcal F$. Then $\mathcal G$ meets every edge of $\mathcal H_{n,d}$. A vertex of level $k_{n,d+1}+i$ lies in
\[
\binom{k_{n,d+1}+i}{i}\binom{n-k_{n,d+1}-i}{d-i}=\left(\frac{\binom di}{2^dd!}+o(1)\right)n^d
\]
edges, so $e(\mathcal H_{n,d})\le\Delta_1(\mathcal H_{n,d})|\mathcal G|$, where $\Delta_1$ is the maximum degree, and Lemma~\ref{lem:2.2} gives $|\mathcal G|\ge(\binom{d}{\lfloor d/2\rfloor}^{-1}+o(1))M_n$. In particular, $\pi_1=1$.

Let $t_d(n)$ be the maximum size of a subset of the discrete cube containing no $d$-dimensional subcube, and define
\[
\lambda_d=\lim_{n\to\infty}\frac{t_d(n)}{2^n}.
\]
This limit exists: averaging over codimension-one faces shows that $t_d(n)/2^n$ is nonincreasing for $n\ge d$. It is the vertex Tur\'an density for subcubes, rather than the extremal density for arbitrary weak poset copies in the full Boolean lattice.

\begin{proposition}\label{prop:cube-turan}
For every $d\ge1$,
\[
(d+1)\lambda_d\le\pi_d
\le d+1-\binom d{\lfloor d/2\rfloor}^{-1}.
\]
There is an absolute constant $c\in(0,1)$ such that
\[
\binom d{\lfloor d/2\rfloor}^{-1}
\le d+1-\pi_d\le(d+1)c^d.
\]
In particular, $d+1-\pi_d=\exp(-\Theta(d))$ as $d\to\infty$.
\end{proposition}

\begin{proof}
Choose a subcube-free family $\mathcal S\subseteq2^{[n]}$ of size $t_d(n)$. For a uniformly random $R\subseteq[n]$, put
$\mathcal S\triangle R=\{A\triangle R:A\in\mathcal S\}$.
Symmetric difference with $R$ is an automorphism of the discrete cube, so $\mathcal S\triangle R$ is still subcube-free. By Lemma~\ref{lem:2.1}, its intersection with $\mathcal W_{n,d+1}$ is weakly $Q_d$-free. Every fixed subset of $[n]$ belongs to $\mathcal S\triangle R$ with probability $t_d(n)/2^n$. Consequently,
\[
\operatorname{ex}_d(n)\ge
\mathbb E|(\mathcal S\triangle R)\cap\mathcal W_{n,d+1}|
=\frac{t_d(n)}{2^n}|\mathcal W_{n,d+1}|.
\]
Letting $n\to\infty$ proves the lower bound on $\pi_d$; its upper bound was proved above. Ellis, Ivan and Leader~\cite[Theorem~7]{Ellis2024Daisies} prove $1-\lambda_d\le c^d$ for an absolute $c<1$. Combining the two inequalities gives the deficit bounds. Finally,
$\binom d{\lfloor d/2\rfloor}=\Theta(2^d/\sqrt d)$,
so these bounds imply the claimed exponential order.
\end{proof}

Here $n\to\infty$ is taken for each fixed $d$ before considering $d\to\infty$. The proposition determines the exponential order of the deficit, not its exponential rate: the displayed lower and upper bounds are not asserted to differ by only a polynomial factor in $d$. For $d=2$, the transfer inequality gives only $\pi_2\ge2$, and the sharper construction of Tompkins is needed.

\section{Supersaturation and codegrees}\label{sec4}

For $\mathcal F\subseteq\mathcal W_{n,d+1}$, let $N_d(\mathcal F)=e(\mathcal H_{n,d}[\mathcal F])$ be the number of canonical subcubes contained in $\mathcal F$.

\begin{proposition}\label{prop:4.1}
For every fixed $d\ge1$ and every $\epsilon>0$ there is $c=c(d,\epsilon)>0$ such that for all sufficiently large $n$, every $\mathcal F\subseteq\mathcal W_{n,d+1}$ with $|\mathcal F|\ge(\pi_d+\epsilon)M_n$ satisfies $N_d(\mathcal F)\ge cn^dM_n$.
\end{proposition}

\begin{proof}
By Proposition~\ref{prop:3.2} and since $\rho_{m,d+1}\to1$, we may choose $m=m(d,\epsilon)$ so large that
\[
\operatorname{ex}_d(m)\le\left(\pi_d+\frac{\epsilon}{8}\right)M_m
\qquad\text{and}\qquad
(1-\rho_{m,d+1})(d+1+\epsilon)\le\frac{\epsilon}{8}.
\]
Let $(B,Y)$ be a uniformly random centred $m$-section and let $X_{B,Y}=|\mathcal F\cap\mathcal W(B,Y)|$. Since $\pi_d\le d+1$, Lemma~\ref{lem:3.1} and the choice of $m$ give
\[
\mathbb EX_{B,Y}\ge\rho_{m,d+1}(\pi_d+\epsilon)M_m\ge\left(\pi_d+\frac{7\epsilon}{8}\right)M_m.
\]
Call a section rich if $X_{B,Y}>\operatorname{ex}_d(m)$, and let $q$ be the proportion of rich sections. Every rich section contains a canonical subcube lying in $\mathcal F$. Since every section has at most $(d+1)M_m$ vertices,
\[
\left(\pi_d+\frac{7\epsilon}{8}\right)M_m\le\mathbb EX_{B,Y}\le\operatorname{ex}_d(m)+q(d+1)M_m\le\left(\pi_d+\frac{\epsilon}{8}\right)M_m+q(d+1)M_m,
\]
and hence $q\ge3\epsilon/(4(d+1))=:\beta$.

By symmetry, every edge of $\mathcal H_{n,d}$ lies in the same number, say $s_{n,m}$, of centred $m$-sections. Counting pairs consisting of a section and an edge contained in it, we obtain
\[
|\mathfrak S_{n,m}|\,e(\mathcal H_{m,d})=s_{n,m}\,e(\mathcal H_{n,d}).
\]
On the other hand, every rich section contains at least one edge of $\mathcal H_{n,d}[\mathcal F]$, and every such edge lies in exactly $s_{n,m}$ sections, so $s_{n,m}N_d(\mathcal F)\ge\beta|\mathfrak S_{n,m}|$. Combining the last two displays, we get
\[
N_d(\mathcal F)\ge\beta\frac{e(\mathcal H_{n,d})}{e(\mathcal H_{m,d})}.
\]
Here $e(\mathcal H_{m,d})$ depends only on $d$ and $\epsilon$, and $e(\mathcal H_{n,d})=\Theta_d(n^dM_n)$ by Lemma~\ref{lem:2.2}. This completes the proof.
\end{proof}

Since $e(\mathcal H_{n,d})=\Theta_d(n^dM_n)$, Proposition~\ref{prop:4.1} says that a family of size at least $(\pi_d+\epsilon)M_n$ contains a positive proportion, depending on $d$ and $\epsilon$, of all edges of $\mathcal H_{n,d}$. Propositions~\ref{prop:3.2} and~\ref{prop:4.1} together prove Theorem~\ref{thm:1.1}.

We next compute the codegrees of $\mathcal H_{n,d}$, which are needed for the container method. For an $r$-uniform hypergraph $\mathcal H$ and $1\le j\le r$, let $\Delta_j(\mathcal H)$ be the maximum number of edges of $\mathcal H$ containing a fixed set of $j$ vertices.

\begin{lemma}\label{lem:5.1}
Let $d\ge1$ and $r=2^d$. For every $2\le j\le r$,
\[
\Delta_j(\mathcal H_{n,d})=O_d\left(n^{d-\lceil\log_2j\rceil}\right).
\]
In particular, $\Delta_r(\mathcal H_{n,d})=1$.
\end{lemma}

\begin{proof}
Let $F_1,\dots,F_j$ be distinct vertices of $\mathcal H_{n,d}$ contained in a common canonical subcube, and let
\[
I=\bigcap_{a=1}^{j}F_a,\qquad U=\bigcup_{a=1}^{j}F_a,\qquad t=|U\setminus I|.
\]
Inside a canonical subcube containing all $F_a$, the sets $F_1,\dots,F_j$ agree outside $U\setminus I$, so they are determined by their intersections with $U\setminus I$. Thus $j\le2^t$, that is, $t\ge\lceil\log_2j\rceil$.

Now let $\mathcal Q(A,X)$ be a canonical subcube containing all $F_a$. Then $A\subseteq I$, $U\subseteq A\cup X$, $|A|=k_{n,d+1}$ and $|X|=d$. Let $P=I\setminus A$ and $Q=(A\cup X)\setminus U$; see Figure~\ref{fig:blocks}.

\begin{figure}[htbp]
\centering
\begin{tikzpicture}[x=1.2cm, y=1cm]
\draw (0,0) rectangle (4,2);
\draw (1,0) -- (1,2);
\draw (2,0) -- (2,2);
\draw (3,0) -- (3,2);
\node at (0.5,1) {$A$};
\node at (1.5,1) {$P$};
\node at (2.5,1) {$U \setminus I$};
\node at (3.5,1) {$Q$};
\draw[<->] (0,-0.25) -- (2,-0.25) node[midway,below] {$I$};
\draw[<->] (0,-0.75) -- (3,-0.75) node[midway,below] {$U$};
\draw[<->] (1,-1.25) -- (4,-1.25) node[midway,below] {$X$};
\draw[<->] (2,2.25) -- (3,2.25) node[midway,above] {$t$};
\draw[<->] (1,2.75) -- (4,2.75) node[midway,above] {$d$};
\end{tikzpicture}
\caption{The four blocks $A$, $P$, $U\setminus I$ and $Q$.}
\label{fig:blocks}
\end{figure}
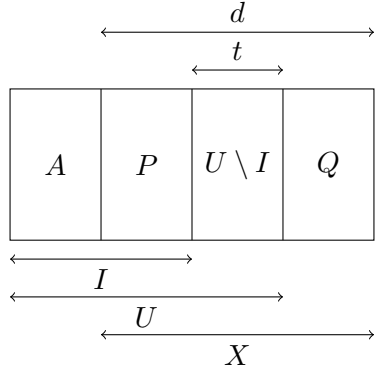

Since $A\subseteq I\subseteq U\subseteq A\cup X$, we have $X=(U\setminus I)\cup P\cup Q$ with the three parts disjoint, so $|P|+|Q|=d-t$. Moreover, $|P|=|I|-k_{n,d+1}$ is determined by $F_1,\dots,F_j$, and hence so is $|Q|$. The subcube $\mathcal Q(A,X)$ is determined by $P\subseteq I$ and $Q\subseteq[n]\setminus U$, so the number of canonical subcubes containing all $F_a$ is at most
\[
\binom{|I|}{|P|}\binom{n-|U|}{|Q|}\le n^{|P|+|Q|}=n^{d-t}.
\]
Since $t\ge\lceil\log_2j\rceil$, the bound follows. If $j=2^d$, then $t=d$, $P=Q=\emptyset$ and the subcube is unique.
\end{proof}

The next inequality explains the exponent in Theorem~\ref{thm:1.3}.

\begin{lemma}\label{lem:5.2}
Let $r=2^d$ and $\gamma_d=d/(2^d-1)$. For every $2\le j\le r$,
\[
\lceil\log_2j\rceil\ge\gamma_d(j-1),
\]
with equality only if $j=r$.
\end{lemma}

\begin{proof}
Let $t=\lceil\log_2j\rceil$, so that $1\le t\le d$ and $j-1\le2^t-1$. The sequence $t/(2^t-1)$ is strictly decreasing in $t$, so $\gamma_d=d/(2^d-1)\le t/(2^t-1)$. Hence
\[
\gamma_d(j-1)\le\gamma_d(2^t-1)\le t.
\]
If equality holds throughout, then $t=d$ and $j=2^t=r$.
\end{proof}

Thus, if $\mathcal F\subseteq\mathcal W_{n,d+1}$ has $|\mathcal F|\ge(\pi_d+\epsilon)M_n$, then $e(\mathcal H_{n,d}[\mathcal F])=\Omega_{d,\epsilon}(n^dM_n)$ by Proposition~\ref{prop:4.1}, while $\Delta_j(\mathcal H_{n,d}[\mathcal F])=O_d(n^{d-\lceil\log_2j\rceil})$ by Lemma~\ref{lem:5.1}. This is exactly the situation in which the hypergraph container method applies.

\section{Containers and enumeration}\label{sec6}

We use the container theorem of Saxton and Thomason~\cite{Saxton2012HypergraphC} in the form with fingerprints. We first recall their codegree function. Let $\mathcal H$ be an $r$-uniform hypergraph with $v$ vertices and average degree $D>0$. For $x\in V(\mathcal H)$ and $2\le j\le r$, let
\[
d^{(j)}(x)=\max\{d_{\mathcal H}(\sigma):x\in\sigma\subseteq V(\mathcal H),\ |\sigma|=j\},
\qquad
\delta_j=\frac{\sum_xd^{(j)}(x)}{\tau^{j-1}vD},
\]
where $d_{\mathcal H}(\sigma)$ is the number of edges containing $\sigma$, and let
\[
\delta(\mathcal H,\tau)=2^{\binom r2-1}\sum_{j=2}^{r}2^{-\binom{j-1}{2}}\delta_j.
\]
Note that $\delta_j\le\Delta_j(\mathcal H)/(\tau^{j-1}D)$. The following is the part of Corollary~3.6 of \cite{Saxton2012HypergraphC} concerning independent sets.

\begin{theorem}\label{container_theorem}
For every $r\ge2$ there is a constant $a_r>0$ such that the following holds. Let $\mathcal H$ be an $r$-uniform hypergraph with $v$ vertices and positive average degree, and let $0<\xi,\tau<1/2$ satisfy $\delta(\mathcal H,\tau)\le\xi/(12r!)$. Then there are an integer $b=b(r,\xi)$ and a function $\mathcal C:\mathcal P(V(\mathcal H))^b\to\mathcal P(V(\mathcal H))$ such that for every independent set $I$ of $\mathcal H$ there is an ordered tuple $\mathbf T=(T_1,\dots,T_b)$ with
\[
T_i\subseteq I,\qquad |T_i|\le a_r\tau v\quad(1\le i\le b),\qquad I\subseteq\mathcal C(\mathbf T),\qquad e(\mathcal H[\mathcal C(\mathbf T)])\le\xi e(\mathcal H).
\]
\end{theorem}

In \cite{Saxton2012HypergraphC} the number of sets in the tuple is $O_r(\log(1/\xi))$, and we append empty sets so that every tuple has length $b$. We shall use that the container $\mathcal C(\mathbf T)$ is determined by the ordered tuple $\mathbf T$.

In this and the next section, $r=2^d$ and $\gamma=\gamma_d=d/(2^d-1)$.

\begin{proposition}\label{prop:Q_d containers}
For every $\eta>0$ there are an integer $b=b(d,\eta)$ and a constant $C=C(d,\eta)>0$ such that the following holds for all sufficiently large $n$. For every independent set $I$ of $\mathcal H_{n,d}$ there are an ordered tuple $\mathbf T=(T_1,\dots,T_b)$ and a set $R=R(\mathbf T)$ such that, with $T=\bigcup_{i=1}^{b}T_i$,
\begin{enumerate}
\item[\textup{(i)}] $T\subseteq I\subseteq T\cup R$ and $T\cap R=\emptyset$;
\item[\textup{(ii)}] $\sum_{i=1}^{b}|T_i|\le Cn^{-\gamma}M_n$;
\item[\textup{(iii)}] $|R|\le(\pi_d+\eta)M_n$;
\item[\textup{(iv)}] $R$ is determined by $\mathbf T$.
\end{enumerate}
\end{proposition}

\begin{proof}
By Proposition~\ref{prop:4.1}, there is $c=c(d,\eta)>0$ such that for all sufficiently large $n$, every $U\subseteq\mathcal W_{n,d+1}$ with $|U|\ge(\pi_d+\eta)M_n$ spans at least $cn^dM_n$ edges of $\mathcal H_{n,d}$. Let $\xi=\min\{1/4,c/2\}$.

By Lemma~\ref{lem:2.2}, $e(\mathcal H_{n,d})\le n^dM_n$ for large $n$ and the average degree of $\mathcal H_{n,d}$ is $D_n=\Theta_d(n^d)$. Let $\tau=Kn^{-\gamma}$, where $K$ is a large constant to be chosen. By Lemmas~\ref{lem:5.1} and~\ref{lem:5.2}, for $2\le j\le r$,
\[
\frac{\Delta_j(\mathcal H_{n,d})}{\tau^{j-1}D_n}
=O_d\left(K^{-(j-1)}n^{-\lceil\log_2j\rceil+\gamma(j-1)}\right),
\]
where the exponent of $n$ is negative for $j<r$ and zero for $j=r$. Hence we may choose $K$ sufficiently large, and then $n$ sufficiently large, so that $\delta(\mathcal H_{n,d},\tau)\le\xi/(12r!)$ and $\tau<1/2$. This includes the case $d=1$, in which only $j=r=2$ occurs.

Theorem~\ref{container_theorem} gives an ordered tuple $\mathbf T$ and a container $\mathcal C(\mathbf T)$ with $I\subseteq\mathcal C(\mathbf T)$ and
\[
e(\mathcal H_{n,d}[\mathcal C(\mathbf T)])\le\xi e(\mathcal H_{n,d})<cn^dM_n.
\]
By the choice of $c$, we have $|\mathcal C(\mathbf T)|<(\pi_d+\eta)M_n$. Let $R(\mathbf T)=\mathcal C(\mathbf T)\setminus T$. Then (i), (iii) and (iv) hold, and
\[
\sum_{i=1}^{b}|T_i|\le ba_r\tau|\mathcal W_{n,d+1}|\le ba_rK(d+1)n^{-\gamma}M_n,
\]
which gives (ii).
\end{proof}

We also need a bound on the number of ordered tuples. Let $v_n=|\mathcal W_{n,d+1}|$ and $A=2^b-1$. The number of ordered tuples $(T_1,\dots,T_b)$ of subsets of $\mathcal W_{n,d+1}$ whose union has size $t$ is at most
\[
\binom{v_n}{t}A^t.
\]
Indeed, after choosing the union, every vertex of the union is assigned the non-empty subset of $[b]$ consisting of the indices $i$ with $T_i$ containing it. Note that this bound allows the sets $T_i$ to overlap.

\begin{proof}[Proof of Theorem~\ref{thm:1.2}]
Let $\mathcal F_n\subseteq\mathcal W_{n,d+1}$ be a $Q_d$-free family of size $\operatorname{ex}_d(n)$. Every subfamily of $\mathcal F_n$ is $Q_d$-free, so $\operatorname{forb}_d(n)\ge2^{|\mathcal F_n|}=2^{(\pi_d+o(1))M_n}$.

For the upper bound, fix $\eta>0$ and apply Proposition~\ref{prop:Q_d containers} with $\eta/2$ in place of $\eta$. Let $s_0=\lfloor Cn^{-\gamma}M_n\rfloor$ and $A=2^b-1$. By the bound above, the number of ordered tuples with union of size at most $s_0$ is at most
\[
\sum_{t=0}^{s_0}\binom{v_n}{t}A^t=2^{O_{d,\eta}(s_0\log(v_n/s_0))}=2^{o(M_n)},
\]
where we used $\binom{v_n}{t}\le(ev_n/t)^t$, $v_n=\Theta_d(M_n)$ and $s_0/v_n=\Theta_{d,\eta}(n^{-\gamma})$. For a fixed tuple $\mathbf T$, every independent set with this tuple is a subset of $T\cup R(\mathbf T)$, which has size at most $(\pi_d+\eta/2)M_n+s_0$. Thus
\[
\operatorname{forb}_d(n)\le2^{o(M_n)}2^{(\pi_d+\eta/2)M_n+s_0}=2^{(\pi_d+\eta/2+o(1))M_n}.
\]
Since $\eta>0$ was arbitrary, the upper bound follows.
\end{proof}

\section{The random layer window}\label{sec8}

Let $X=\mathcal W_{n,d+1}(p)$. We first prove part (i) of Theorem~\ref{thm:1.3}, for which we need the following estimate.

\begin{lemma}\label{lem:weighted_fingerprint}
Let $p=p(n)$ and $\tau=\tau(n)$ be positive with $p/\tau\to\infty$, and let $v$ and $s$ be non-negative integers with $v=O(M_n)$ and $s=O(\tau M_n)$. Then for every fixed $A>0$,
\[
\sum_{t=0}^{s}\binom vt(Ap)^t=\exp(o(pM_n)),
\]
where $\binom vt=0$ for $t>v$.
\end{lemma}

\begin{proof}
Let $x=p/\tau$, so $x>1$ for large $n$. Since $t\le s$,
\[
\sum_{t=0}^{s}\binom vt(Ap)^t=\sum_{t=0}^{s}\binom vt(A\tau)^tx^t\le x^s(1+A\tau)^v.
\]
The logarithm of the right-hand side is at most $s\log x+Av\tau=O(\tau M_n(\log x+1))=o(pM_n)$, since $\tau\log(p/\tau)=o(p)$ when $p/\tau\to\infty$.
\end{proof}

\begin{proposition}\label{prop:dense_random}
If $pn^{\gamma_d}\to\infty$, then with high probability $\alpha_d(\mathcal W_{n,d+1}(p))=(\pi_d+o(1))pM_n$.
\end{proposition}

\begin{proof}
Note that $pM_n\to\infty$. For the lower bound, let $\mathcal F_n\subseteq\mathcal W_{n,d+1}$ be a $Q_d$-free family of size $\operatorname{ex}_d(n)=(\pi_d+o(1))M_n$. By the Chernoff bound, $|X\cap\mathcal F_n|=(\pi_d+o(1))pM_n$ with high probability, and $X\cap\mathcal F_n$ is $Q_d$-free.

For the upper bound, fix $\eta>0$ and apply Proposition~\ref{prop:Q_d containers} with $\eta/4$ in place of $\eta$. Let $\tau=n^{-\gamma_d}$, so the total size of the fingerprints is at most $C\tau M_n=o(pM_n)$. Suppose that $X$ contains an independent set $I$ with $|I|\ge(\pi_d+\eta)pM_n$, and let $\mathbf T=(T_1,\dots,T_b)$, $R=R(\mathbf T)$ and $T=\bigcup_{i=1}^{b}T_i$ be as in Proposition~\ref{prop:Q_d containers}. For large $n$ we have $|T|\le\frac{\eta}{4}pM_n$. Since $I\subseteq T\cup R$ and $T\subseteq I\subseteq X$, we get
\[
|X\cap R|\ge|I|-|T|\ge\left(\pi_d+\frac{3\eta}{4}\right)pM_n.
\]
Now fix an admissible ordered tuple $\mathbf T$ arising from Proposition~\ref{prop:Q_d containers}, with union $T$ of size $t$. The set $R=R(\mathbf T)$ is determined by $\mathbf T$, is disjoint from $T$ and satisfies $|R|\le(\pi_d+\frac{\eta}{4})M_n$. Conditional on $T\subseteq X$, the members of $R$ still belong to $X$ independently with probability $p$, so $|X\cap R|$ has the binomial distribution with parameters $|R|$ and $p$, and the Chernoff bound shows that the conditional probability of the last inequality is at most $\exp(-c_{d,\eta}pM_n)$ for some $c_{d,\eta}>0$. The probability that $T\subseteq X$ is $p^t$, and by the bound in Section~\ref{sec6} the number of ordered tuples with union of size $t$ is at most $\binom{v_n}{t}A^t$, where $v_n=O_d(M_n)$ and $A=2^b-1$. Hence, by the union bound and Lemma~\ref{lem:weighted_fingerprint}, the probability that $X$ contains an independent set of size at least $(\pi_d+\eta)pM_n$ is at most
\[
\exp(-c_{d,\eta}pM_n)\sum_{t\le C\tau M_n}\binom{v_n}{t}(Ap)^t
=\exp(-c_{d,\eta}pM_n+o(pM_n))=o(1).
\]
Thus with high probability $\alpha_d(X)\le(\pi_d+\eta)pM_n$. Since $\eta>0$ was arbitrary, letting $\eta\to0$ along a sequence completes the proof.
\end{proof}

\begin{proposition}\label{prop:sparse_regime}
If $pn^{\gamma_d}\to0$ and $pM_n\to\infty$, then with high probability $\alpha_d(\mathcal W_{n,d+1}(p))=(d+1+o(1))pM_n$.
\end{proposition}

\begin{proof}
We follow the deletion argument of \cite[Proposition~5]{Gerbner2021Supersaturation}. By Lemma~\ref{lem:2.2} and the Chernoff bound, $|X|=(d+1+o(1))pM_n$ with high probability. Moreover,
\[
\mathbb E\,e(\mathcal H_{n,d}[X])=p^{2^d}e(\mathcal H_{n,d})=\Theta_d(p^{2^d}n^dM_n)=pM_n\cdot\Theta_d\left((pn^{\gamma_d})^{2^d-1}\right)=o(pM_n),
\]
so $e(\mathcal H_{n,d}[X])=o(pM_n)$ with high probability by Markov's inequality. Deleting one vertex from every edge of $\mathcal H_{n,d}[X]$ removes $o(pM_n)$ vertices and leaves an independent set. Hence $\alpha_d(X)\ge|X|-o(pM_n)=(d+1-o(1))pM_n$, and $\alpha_d(X)\le|X|$ trivially.
\end{proof}

Propositions~\ref{prop:dense_random} and~\ref{prop:sparse_regime} prove Theorem~\ref{thm:1.3}.

\section{Graded posets in their height window}\label{sec:general}

The arguments of Sections~\ref{sec2}--\ref{sec8} extend to other posets, with two differences. A poset need not have a copy in a window of its own height, and when copies exist they may use different numbers of coordinates. We therefore keep copies with different coordinate patterns apart, both in the supersaturation statement and in the container argument.

A finite nonempty poset $P$ is graded of height $h$ if every maximal chain of $P$ has exactly $h$ elements. For $n\ge h-1$, recall that
\[
k_{n,h}=\left\lfloor\frac{n-h+1}{2}\right\rfloor,\qquad
\mathcal W_{n,h}=\bigcup_{i=0}^{h-1}\binom{[n]}{k_{n,h}+i}.
\]
For a graded poset $P$ of height $h$, let $\operatorname{ex}_P(n)$ and $\operatorname{forb}_P(n)$ be the maximum size and the number of $P$-free subfamilies of $\mathcal W_{n,h}$, let $N_P(\mathcal F)$ be the number of copies of $P$ in $\mathcal F$, counted as subfamilies, and let $\alpha_P(X)$ be the maximum size of a $P$-free subfamily of $X$. We call $P$ connected if its Hasse diagram is connected as an undirected graph.

\begin{lemma}\label{gen:rank}
Let $P$ be graded of height $h$. Then $P$ has a rank function $\operatorname{rk}:P\to\{0,\dots,h-1\}$ such that every maximal chain contains exactly one element of each rank, and every injective order-preserving map $\varphi:P\to\mathcal W_{n,h}$ satisfies $|\varphi(x)|=k_{n,h}+\operatorname{rk}(x)$ for all $x\in P$. In particular, if $y$ covers $x$, then $\varphi(y)\setminus\varphi(x)$ has exactly one element.
\end{lemma}

\begin{proof}
Fix $x\in P$. A maximal chain of the down-set of $x$ followed by a maximal chain of the up-set of $x$, with $x$ counted once, is a maximal chain of $P$ and hence has $h$ elements. Thus all maximal chains of the down-set of $x$ have the same number of elements, and we let $\operatorname{rk}(x)$ be this number minus one. The same argument for the up-set shows that every maximal chain of $P$ contains one element of each rank. The image of a maximal chain under $\varphi$ consists of $h$ distinct sets in $h$ consecutive levels that increase strictly, so it contains one set from each level, in order. Since every $x$ lies on a maximal chain, $|\varphi(x)|=k_{n,h}+\operatorname{rk}(x)$.
\end{proof}

\begin{proposition}\label{gen:limit}
For every graded poset $P$ of height $h$, the limit $\pi(P)=\lim_{n\to\infty}\operatorname{ex}_P(n)/M_n$ exists. If $h\ge2$, then $h-1\le\pi(P)\le h$.
\end{proposition}

\begin{proof}
For $n\ge m\ge h-1$, a centred $m$-section of $\mathcal W_{n,h}$ is a pair $(B,Y)$ of disjoint subsets of $[n]$ with $|Y|=m$ and $|B|=k_{n,h}-k_{m,h}$, and its family $\mathcal W(B,Y)$ consists of the sets $B\cup Z$ with $Z\subseteq Y$ and $k_{m,h}\le|Z|\le k_{m,h}+h-1$. The map $Z\mapsto B\cup Z$ is an order isomorphism from $\mathcal W_{m,h}$ onto $\mathcal W(B,Y)$. Let
\[
\rho_{m,h}=\min_{0\le i<h}\frac{\binom{m}{k_{m,h}+i}}{M_m},
\]
so that $\rho_{m,h}\to1$ as $m\to\infty$. The proof of Lemma~\ref{lem:3.1} gives $\mathbb E|\mathcal F\cap\mathcal W(B,Y)|\ge\rho_{m,h}(M_m/M_n)|\mathcal F|$ for a uniformly random section, and the intersection of a $P$-free family with a section is $P$-free. Hence $a_n=\operatorname{ex}_P(n)/M_n$ satisfies $a_n\le a_m/\rho_{m,h}$, and the limit exists exactly as in Proposition~\ref{prop:3.2}. A union of $h-1$ levels contains no chain of $h$ sets and is therefore $P$-free, which gives $\pi(P)\ge h-1$, and $\pi(P)\le h$ follows from $|\mathcal W_{n,h}|=(h+o(1))M_n$.
\end{proof}

Up to Theorem~\ref{gen:connected-main}, let $P$ be connected and graded of height $h\ge2$, and let $r=|P|$. For a copy $\mathcal E=\{F_x:x\in P\}$ of $P$ in $\mathcal W_{n,h}$, let
\[
A=\bigcap_{x\in P}F_x,\qquad X=\Big(\bigcup_{x\in P}F_x\Big)\setminus A,\qquad t=|X|.
\]
Removing $A$ from every member and identifying $X$ with $[t]$, we obtain a family $\mathcal T$ of $r$ distinct subsets of $[t]$ with empty intersection and union $[t]$. We identify two such families if a permutation of $[t]$ maps one onto the other, and we call $\mathcal T$ the type of the copy. Let $\mathfrak T(P)$ be the set of all types of copies of $P$ in $\mathcal W_{n,h}$, over all $n$. For $\mathcal T\in\mathfrak T(P)$ let $t(\mathcal T)=t$ and let $q(\mathcal T)$ be the size of a smallest member of $\mathcal T$; by Lemma~\ref{gen:rank} the members of $\mathcal T$ have sizes $q(\mathcal T),\dots,q(\mathcal T)+h-1$. The type depends only on the image family, so additional comparabilities inside a weak copy, or several embeddings of $P$ with the same image, do not produce further types or further copies.

\begin{lemma}\label{gen:finite}
Every $\mathcal T\in\mathfrak T(P)$ satisfies $h-1\le t(\mathcal T)\le r-1$. Hence $\mathfrak T(P)$ is finite, and every type occurs in $\mathcal W_{n,h}$ for all sufficiently large $n$.
\end{lemma}

\begin{proof}
Take a spanning tree of the Hasse diagram of $P$. By Lemma~\ref{gen:rank}, the images of the two ends of each of its $r-1$ edges differ in exactly one coordinate, and a coordinate that is not constant on the copy changes along the tree path between two members on which it differs. Hence at most $r-1$ coordinates vary. The image of a maximal chain changes $h-1$ distinct coordinates, which gives the lower bound. There are finitely many families of $r$ subsets of a set of size at most $r-1$, so $\mathfrak T(P)$ is finite. Finally, a type $\mathcal T$ is realised in $\mathcal W_{n,h}$ by adjoining a common set of size $k_{n,h}-q(\mathcal T)$ to its members and choosing $t(\mathcal T)$ further coordinates, which is possible for all large $n$.
\end{proof}

The set $\mathfrak T(P)$ may be empty. Then $P$ has no copy in any height window, so $\operatorname{ex}_P(n)=|\mathcal W_{n,h}|$, $\operatorname{forb}_P(n)=2^{|\mathcal W_{n,h}|}$, $\alpha_P(X)=|X|$ for every $X\subseteq\mathcal W_{n,h}$, and $\pi(P)=h$. Examples are given in Section~\ref{sec:examples}. Until Theorem~\ref{gen:connected-main} we assume that $\mathfrak T(P)\ne\emptyset$.

For $\mathcal T\in\mathfrak T(P)$ let $\mathcal H_{\mathcal T}(n)$ be the $r$-uniform hypergraph on $\mathcal W_{n,h}$ whose edges are the copies of $P$ of type $\mathcal T$, let $e_{\mathcal T}(n)=e(\mathcal H_{\mathcal T}(n))$, and let $N_{\mathcal T}(\mathcal F)=e(\mathcal H_{\mathcal T}(n)[\mathcal F])$. These hypergraphs partition the copies of $P$. For $1\le j\le r$ let
\[
s_j(\mathcal T)=\min\Big\{\Big|\bigcup\mathcal J\setminus\bigcap\mathcal J\Big|:\mathcal J\subseteq\mathcal T,\ |\mathcal J|=j\Big\}.
\]
Note that $s_1(\mathcal T)=0$, $s_r(\mathcal T)=t(\mathcal T)$, and $s_j(\mathcal T)\ge\lceil\log_2j\rceil$ for $j\ge2$, since $j$ distinct sets that agree outside a set of size $s$ satisfy $j\le2^s$.

\begin{lemma}\label{gen:degrees}
Let $\mathcal T\in\mathfrak T(P)$, let $t=t(\mathcal T)$ and $q=q(\mathcal T)$, and let $a(\mathcal T)$ be the number of permutations of $[t]$ that map $\mathcal T$ onto itself. For all sufficiently large $n$,
\[
e_{\mathcal T}(n)=\binom{n}{k_{n,h}-q}\binom{n-k_{n,h}+q}{t}\frac{t!}{a(\mathcal T)}
=\left(\frac{1}{2^ta(\mathcal T)}+o(1)\right)n^tM_n,
\]
and $\Delta_j(\mathcal H_{\mathcal T}(n))=\Theta_P(n^{t-s_j(\mathcal T)})$ for $1\le j\le r$. In particular, the average degree of $\mathcal H_{\mathcal T}(n)$ is $\Theta_P(n^t)$.
\end{lemma}

\begin{proof}
The common intersection of a copy of type $\mathcal T$ has size $k_{n,h}-q$. After choosing it and the $t$ varying coordinates, there are exactly $t!/a(\mathcal T)$ image families of type $\mathcal T$. This gives the formula, and the asymptotic form follows from $\binom{n}{k_{n,h}-q}=(1+o(1))M_n$ and $n-k_{n,h}+q=n/2+O_P(1)$.

For the upper bound on $\Delta_j$, fix $j$ sets contained in a common edge of type $\mathcal T$ and let $I$ and $U$ be their intersection and union. The $j$ sets correspond to a subfamily $\mathcal J\subseteq\mathcal T$ with $|\mathcal J|=j$, in one of boundedly many ways, and the coordinates of $[t]$ on which $\mathcal J$ is not constant are mapped bijectively onto $U\setminus I$, in at most $t!$ ways. Suppose that $u$ coordinates of $[t]$ lie in every member of $\mathcal J$ and $v$ lie in no member of $\mathcal J$. The former are mapped into $I$ and the latter into $[n]\setminus U$, and once these images are chosen the edge is determined. Since $u+v=t-|\bigcup\mathcal J\setminus\bigcap\mathcal J|\le t-s_j(\mathcal T)$, the number of edges containing the $j$ sets is $O_P(n^{t-s_j(\mathcal T)})$.

For the lower bound, choose $\mathcal J$ attaining $s_j(\mathcal T)$ and fix the image of $\mathcal J$ inside one copy of type $\mathcal T$ in $\mathcal W_{n,h}$. Then $|I|=n/2+O_P(1)$ and $n-|U|=n/2+O_P(1)$, so there are $\Theta_P(n^{u+v})$ choices for the images of the remaining $u+v=t-s_j(\mathcal T)$ coordinates, and distinct choices give distinct edges up to a bounded multiplicity.
\end{proof}

For $\mathcal T\in\mathfrak T(P)$ let
\[
\nu_{\mathcal T}=\min_{2\le j\le r}\frac{s_j(\mathcal T)}{j-1},
\]
which is the minimum of $|\bigcup\mathcal J\setminus\bigcap\mathcal J|/(|\mathcal J|-1)$ over all subfamilies $\mathcal J\subseteq\mathcal T$ with at least two members, and let
\[
x_-(P)=\min_{\mathcal T}t(\mathcal T),\qquad
x_+(P)=\max_{\mathcal T}t(\mathcal T),\qquad
\kappa(P)=\frac{x_+(P)}{r-1},
\]
\[
\nu(P)=\min_{\mathcal T}\nu_{\mathcal T},\qquad
\nu_+(P)=\max_{\mathcal T}\nu_{\mathcal T},
\]
where $\mathcal T$ ranges over $\mathfrak T(P)$. These are positive rationals determined by $P$, and
\[
0<\nu(P)\le\nu_+(P)\le\kappa(P)\le1,\qquad\nu(P)\le\frac{x_-(P)}{r-1}.
\]
Indeed $\nu_{\mathcal T}\le s_r(\mathcal T)/(r-1)=t(\mathcal T)/(r-1)$, and $t(\mathcal T)\le r-1$ by Lemma~\ref{gen:finite}. For $Q_d$ there is a single type, namely $2^{[d]}$, and Lemmas~\ref{lem:5.1} and~\ref{lem:5.2} give $s_j=\lceil\log_2j\rceil$ and $\nu(Q_d)=\nu_+(Q_d)=\kappa(Q_d)=d/(2^d-1)$.

\begin{proposition}\label{gen:weighted-ss}
For every $\epsilon>0$ there is $c=c(P,\epsilon)>0$ such that, for all sufficiently large $n$, every $\mathcal F\subseteq\mathcal W_{n,h}$ with $|\mathcal F|\ge(\pi(P)+\epsilon)M_n$ satisfies
\[
\sum_{\mathcal T\in\mathfrak T(P)}\frac{N_{\mathcal T}(\mathcal F)}{e_{\mathcal T}(n)}\ge c.
\]
In particular $N_P(\mathcal F)=\Omega_{P,\epsilon}(n^{x_-(P)}M_n)$. Moreover, $\pi(P)<h$.
\end{proposition}

\begin{proof}
Choose $m$ so large that every type occurs in $\mathcal W_{m,h}$, that $\operatorname{ex}_P(m)\le(\pi(P)+\epsilon/4)M_m$ and that $(1-\rho_{m,h})(h+\epsilon)\le\epsilon/4$. As in the proof of Proposition~\ref{prop:4.1}, the intersection of $\mathcal F$ with a uniformly random centred $m$-section has expected size at least $(\pi(P)+3\epsilon/4)M_m$. Call a section rich if its intersection with $\mathcal F$ has more than $\operatorname{ex}_P(m)$ members. Since $|\mathcal W_{m,h}|\le hM_m$, the proportion of rich sections is at least $\beta=\epsilon/(2h)$.

Permutations of $[n]$ act transitively on the edges of a fixed type, so every edge of type $\mathcal T$ lies in the same number $b_{\mathcal T}$ of sections, and counting pairs consisting of a section and an edge of type $\mathcal T$ inside it gives $|\mathfrak S_{n,m}|e_{\mathcal T}(m)=b_{\mathcal T}e_{\mathcal T}(n)$. Every rich section contains a copy of $P$ inside $\mathcal F$, of some type. Hence $\sum_{\mathcal T}b_{\mathcal T}N_{\mathcal T}(\mathcal F)\ge\beta|\mathfrak S_{n,m}|$, that is,
\[
\sum_{\mathcal T\in\mathfrak T(P)}e_{\mathcal T}(m)\frac{N_{\mathcal T}(\mathcal F)}{e_{\mathcal T}(n)}\ge\beta.
\]
Since $m$ depends only on $P$ and $\epsilon$, the first assertion follows with $c=\beta/\max_{\mathcal T}e_{\mathcal T}(m)$, and the second follows from Lemma~\ref{gen:degrees}, since $e_{\mathcal T}(n)=\Omega_P(n^{x_-(P)}M_n)$ for every type.

Finally, fix a type $\mathcal T$. The complement of a $P$-free family meets every edge of $\mathcal H_{\mathcal T}(n)$, so it has at least $e_{\mathcal T}(n)/\Delta_1(\mathcal H_{\mathcal T}(n))=\Omega_P(M_n)$ members by Lemma~\ref{gen:degrees}. Hence $\pi(P)<h$.
\end{proof}

Unlike the Boolean case, a rich section need not contain a copy of maximum dimension, so the weighted form is the natural conclusion. It is also the form that the container argument needs.

\begin{proposition}\label{gen:containers}
For every $\eta>0$ there are $C=C(P,\eta)>0$ and $b=b(P,\eta)$ such that, for all sufficiently large $n$, every $P$-free $I\subseteq\mathcal W_{n,h}$ admits an ordered tuple $\mathbf T=(T_1,\dots,T_b)$ and a set $R=R(\mathbf T)$ determined by $\mathbf T$ such that, with $T=\bigcup_{i=1}^bT_i$,
\[
T\subseteq I\subseteq T\cup R,\qquad T\cap R=\emptyset,\qquad
\sum_{i=1}^b|T_i|\le Cn^{-\nu(P)}M_n,\qquad |R|\le(\pi(P)+\eta)M_n.
\]
\end{proposition}

\begin{proof}
Let $L=|\mathfrak T(P)|$, let $c=c(P,\eta)$ be given by Proposition~\ref{gen:weighted-ss}, let $\xi=\min\{1/4,c/(2L)\}$ and let $\tau=Kn^{-\nu(P)}$ with $K$ a large constant. For every type, Lemma~\ref{gen:degrees} gives
\[
\frac{\Delta_j(\mathcal H_{\mathcal T}(n))}{\tau^{j-1}D_{\mathcal T}(n)}
=O_P\left(K^{-(j-1)}n^{-s_j(\mathcal T)+\nu(P)(j-1)}\right)\qquad(2\le j\le r),
\]
where $D_{\mathcal T}(n)$ is the average degree of $\mathcal H_{\mathcal T}(n)$, and the exponent of $n$ is at most $0$ since $\nu(P)\le\nu_{\mathcal T}$. As there are finitely many types and values of $j$, choosing $K$ large and then $n$ large makes the hypothesis of Theorem~\ref{container_theorem} hold for every $\mathcal H_{\mathcal T}(n)$, with $\tau<1/2$.

A $P$-free family is independent in every $\mathcal H_{\mathcal T}(n)$. Apply Theorem~\ref{container_theorem} to each of these hypergraphs, concatenate the resulting tuples into one ordered tuple $\mathbf T$ of bounded length, and let $\mathcal C$ be the intersection of the resulting containers. Then $I\subseteq\mathcal C$ and $N_{\mathcal T}(\mathcal C)\le\xi e_{\mathcal T}(n)$ for every type. If $|\mathcal C|\ge(\pi(P)+\eta)M_n$, then Proposition~\ref{gen:weighted-ss} gives $\sum_{\mathcal T}N_{\mathcal T}(\mathcal C)/e_{\mathcal T}(n)\ge c>L\xi$, a contradiction. Every $T_i$ is contained in $I$, and the total size of the tuple is $O_{P,\eta}(\tau|\mathcal W_{n,h}|)=O_{P,\eta}(n^{-\nu(P)}M_n)$. Let $R=\mathcal C\setminus T$.
\end{proof}

\begin{theorem}\label{gen:connected-main}
Let $P$ be a connected graded poset of height $h\ge2$ with $\mathfrak T(P)\ne\emptyset$. Then $\operatorname{forb}_P(n)=2^{(\pi(P)+o(1))M_n}$. Moreover, if $X=\mathcal W_{n,h}(p)$ is obtained by retaining every member of $\mathcal W_{n,h}$ independently with probability $p$, then the following hold.
\begin{enumerate}
\item[\textup{(i)}] If $pn^{\nu(P)}\to\infty$, then with high probability $\alpha_P(X)=(\pi(P)+o(1))pM_n$.
\item[\textup{(ii)}] If $pn^{\nu_+(P)}\to0$ and $pM_n\to\infty$, then with high probability $\alpha_P(X)=(h+o(1))pM_n$.
\end{enumerate}
\end{theorem}

\begin{proof}
An extremal $P$-free family and its subfamilies give the lower bound in the counting formula. For the upper bound, apply Proposition~\ref{gen:containers} with $\eta/2$ in place of $\eta$. With $v=|\mathcal W_{n,h}|$, $A=2^b-1$ and $s=O_{P,\eta}(n^{-\nu(P)}M_n)$, the number of possible tuples is at most $\sum_{u\le s}\binom vuA^u=2^{O_{P,\eta}(s\log(v/s))}=2^{o(M_n)}$, and the independent sets with a given tuple lie in a set of size at most $(\pi(P)+\eta/2)M_n+s$. Letting $\eta\to0$ proves the formula.

For (i), the intersection of $X$ with an extremal family gives the lower bound by the Chernoff bound. For the upper bound, fix $\eta>0$, apply Proposition~\ref{gen:containers} with $\eta/4$ in place of $\eta$, and let $\tau=n^{-\nu(P)}$. The proof of Proposition~\ref{prop:dense_random} applies word for word: an independent set of size at least $(\pi(P)+\eta)pM_n$ has fingerprint union $T$ of size $o(pM_n)$ and hence $|X\cap R(\mathbf T)|\ge(\pi(P)+3\eta/4)pM_n$; for an admissible tuple, $R(\mathbf T)$ is determined, disjoint from $T$ and of size at most $(\pi(P)+\eta/4)M_n$, so conditional on $T\subseteq X$ the Chernoff bound gives probability at most $\exp(-c_{P,\eta}pM_n)$; and the union bound over admissible tuples together with Lemma~\ref{lem:weighted_fingerprint} gives
\[
\mathbb P\big(\alpha_P(X)\ge(\pi(P)+\eta)pM_n\big)\le\exp(-c_{P,\eta}pM_n)\sum_{u\le C\tau M_n}\binom vu(Ap)^u=o(1).
\]

For (ii), for every type $\mathcal T$ choose a subfamily $\mathcal J_{\mathcal T}\subseteq\mathcal T$ with $j_{\mathcal T}\ge2$ members and $|\bigcup\mathcal J_{\mathcal T}\setminus\bigcap\mathcal J_{\mathcal T}|=s_{\mathcal T}$ such that $s_{\mathcal T}/(j_{\mathcal T}-1)=\nu_{\mathcal T}$. An image of $\mathcal J_{\mathcal T}$ in $\mathcal W_{n,h}$, in the levels prescribed by its position in $\mathcal T$, is determined by its common intersection, which has size $k_{n,h}-q(\mathcal T)+|\bigcap\mathcal J_{\mathcal T}|$, and by its $s_{\mathcal T}$ varying coordinates, so there are $O_P(n^{s_{\mathcal T}}M_n)$ such images. The expected number of them inside $X$ is therefore
\[
O_P(p^{j_{\mathcal T}}n^{s_{\mathcal T}}M_n)=pM_n\cdot O_P\big((pn^{\nu_{\mathcal T}})^{j_{\mathcal T}-1}\big)=o(pM_n),
\]
since $\nu_{\mathcal T}\le\nu_+(P)$. Summing over the finitely many types and applying Markov's inequality, with high probability we can delete one set from every image of every $\mathcal J_{\mathcal T}$ inside $X$ by removing $o(pM_n)$ sets. Every copy of $P$ of type $\mathcal T$ contains an image of $\mathcal J_{\mathcal T}$, so the remaining family is $P$-free. Since $|X|=(h+o(1))pM_n$ with high probability, (ii) follows.
\end{proof}

The same results hold for induced copies: keep only the types whose families induce $P$, define $\pi_{\mathrm{ind}}(P)$ by induced-$P$-free families, and take the minimum and maximum of $\nu_{\mathcal T}$ over the remaining types. The proofs apply without change, and an empty set of induced types again means that the whole window is free. The reduction to components in Section~\ref{sec:examples} concerns weak copies only.

The two regimes in Theorem~\ref{gen:connected-main} meet when $\nu(P)=\nu_+(P)$, and then $n^{-\nu(P)}$ is the threshold order, since $\pi(P)<h$. This holds in particular whenever $P$ has a single type. Deleting one set from every whole copy of $P$, as in Proposition~\ref{prop:sparse_regime}, gives only the sparse regime $pn^{\kappa(P)}\to0$, because the number of copies is $\Theta_P(n^{x_+(P)}M_n)$; one checks that $\nu(P)=\kappa(P)$ if and only if all types have the same dimension $t$ and $s_j(\mathcal T)\ge t(j-1)/(r-1)$ for every type and every $2\le j\le r$. Deleting subfamilies instead of whole copies is what gives matching exponents for every single-type poset, whether or not this condition holds. When $\nu(P)<\nu_+(P)$, the theorem gives two sufficient conditions and says nothing about the range in between, and the exponent $\nu(P)$ need not be the truth: it is only what the codegree bounds of Lemma~\ref{gen:degrees} give. Section~\ref{sec:examples} contains a tree poset for which $\nu(P)<\nu_+(P)$ while the results of \cite{Jiang2025Tree} show that the dense regime extends down to $pn\to\infty$.

\section{Disconnected posets, examples and the parameter $\mu(P)$}\label{sec:examples}

We first remove the connectedness assumption. Let $P$ be a graded poset of height $h\ge2$ with connected components $P_1,\dots,P_c$, and let $r=|P|$. Each component is graded of height $h$.

\begin{lemma}\label{gen:components}
For every finite family $X$,
\[
\max_{1\le i\le c}\alpha_{P_i}(X)\le\alpha_P(X)\le\max_{1\le i\le c}\alpha_{P_i}(X)+r,
\]
and, writing $\operatorname{forb}_{P_i}(X)$ for the number of $P_i$-free subfamilies of $X$,
\[
\operatorname{forb}_P(X)\le\Big(\sum_{j=0}^{r}\binom{|X|}{j}\Big)\sum_{i=1}^{c}\operatorname{forb}_{P_i}(X).
\]
\end{lemma}

\begin{proof}
Every $P_i$-free family is $P$-free, which gives the lower bound. Let $I\subseteq X$ be $P$-free. Choose copies of $P_1,\dots,P_c$ in $I$ one after another, deleting the members of each chosen copy before choosing the next. If every choice succeeds, the union of the chosen copies is a weak copy of $P$, since comparabilities between members of different components are allowed in a weak copy. Hence the procedure fails at some $P_i$, and at that moment fewer than $r$ members have been deleted and the remaining family is $P_i$-free. This proves the upper bound, and it also writes every $P$-free family as the union of at most $r$ members of $X$ and a $P_i$-free family for some $i$, which gives the counting inequality.
\end{proof}

\begin{corollary}\label{gen:disconnected}
Let $P$ be a graded poset of height $h\ge2$ with components $P_1,\dots,P_c$. Then $\pi(P)=\max_i\pi(P_i)$ and $\operatorname{forb}_P(n)=2^{(\pi(P)+o(1))M_n}$. If some $\mathfrak T(P_i)$ is empty, then no height window contains $P$ and every subfamily is $P$-free. Otherwise let $\nu_*(P)=\min_i\nu(P_i)$ and $\nu_{+,*}(P)=\min_i\nu_+(P_i)$. For every $\epsilon>0$, every $\mathcal F\subseteq\mathcal W_{n,h}$ with $|\mathcal F|\ge(\pi(P)+\epsilon)M_n$ contains $\Omega_{P,\epsilon}(M_n^{c}n^{\sum_ix_-(P_i)})$ copies of $P$ for large $n$, and for $X=\mathcal W_{n,h}(p)$ the following hold.
\begin{enumerate}
\item[\textup{(i)}] If $pn^{\nu_*(P)}\to\infty$, then with high probability $\alpha_P(X)=(\pi(P)+o(1))pM_n$.
\item[\textup{(ii)}] If $pn^{\nu_{+,*}(P)}\to0$ and $pM_n\to\infty$, then with high probability $\alpha_P(X)=(h+o(1))pM_n$.
\end{enumerate}
\end{corollary}

\begin{proof}
The formula for $\pi(P)$ follows from Lemma~\ref{gen:components} after dividing by $M_n$. In the counting inequality the first factor has logarithm $O_P(\log|\mathcal W_{n,h}|)=O_P(n)=o(M_n)$, so Theorem~\ref{gen:connected-main}, or the exact formula for a component with no type, gives the upper bound on $\operatorname{forb}_P(n)$, and an extremal family gives the lower bound.

For supersaturation, apply Proposition~\ref{gen:weighted-ss} to each component, using $\pi(P)\ge\pi(P_i)$: for every $i$ there is a type $\mathcal T_i$ of $P_i$ with $N_{\mathcal T_i}(\mathcal F)=\Omega_{P,\epsilon}(n^{t(\mathcal T_i)}M_n)$. By the bound on $\Delta_1$ in Lemma~\ref{gen:degrees}, at most $O_P(n^{t(\mathcal T_i)})$ copies of type $\mathcal T_i$ meet a fixed set of at most $r$ members, so we can choose copies of $P_1,\dots,P_c$ in $\mathcal F$ one after another, pairwise disjoint, in $\Omega_{P,\epsilon}(M_n^cn^{\sum_it(\mathcal T_i)})$ ways. Their union is a weak copy of $P$, every such copy arises boundedly many times, and $t(\mathcal T_i)\ge x_-(P_i)$.

In (i), the hypothesis gives $pn^{\nu(P_i)}\to\infty$ for every $i$. Apply Theorem~\ref{gen:connected-main} to every component and then Lemma~\ref{gen:components}, whose additive constant is $o(pM_n)$. In (ii), choose $i$ with $\nu_+(P_i)=\nu_{+,*}(P)$. Theorem~\ref{gen:connected-main} gives a $P_i$-free, hence $P$-free, subfamily of $X$ of size $(h+o(1))pM_n$, and $|X|=(h+o(1))pM_n$ gives the matching upper bound.
\end{proof}

A graded poset of height one is an antichain. If it has $r$ elements, then its weak copies are the $r$-element subfamilies, so with $v_n=|\mathcal W_{n,1}|=M_n$ we have $\operatorname{ex}_P(n)=\min\{r-1,v_n\}$, $\operatorname{forb}_P(n)=\sum_{j=0}^{r-1}\binom{v_n}{j}$ and $\alpha_P(X)=\min\{r-1,|X|\}$. Thus $\pi(P)=0$ and $\operatorname{forb}_P(n)=2^{o(M_n)}$.

We now compute the parameters of Section~\ref{sec:general} for some posets. If $P$ has a single type $\mathcal T$, with $t=t(\mathcal T)$, then
\[
x_-(P)=x_+(P)=t,\qquad\Delta_j(\mathcal H_P)=\Theta_P(n^{t-s_j(\mathcal T)}),\qquad
\nu(P)=\nu_+(P)=\min_{2\le j\le r}\frac{s_j(\mathcal T)}{j-1}.
\]
Let $C_h$ be the chain with $h$ elements, let $V_s$ be the poset with one minimal element below $s$ pairwise incomparable maximal elements, let $O_6$ be the crown with six elements, whose Hasse diagram is a cycle of length six, and let $Q_2^+$ be obtained from $Q_2$ by adding a new element above all elements of $Q_2$.

\begin{center}
\begin{tabular}{ccccc}
\toprule
$P$ & $h$ & $x_-(P)=x_+(P)$ & $\nu(P)=\nu_+(P)$ & $\kappa(P)$\\ \midrule
$C_h$, $h\ge2$ & $h$ & $h-1$ & $1$ & $1$\\
$V_s$, $s\ge1$ & $2$ & $s$ & $1$ & $1$\\
$O_6$ & $2$ & $3$ & $3/5$ & $3/5$\\
$Q_d$ & $d+1$ & $d$ & $d/(2^d-1)$ & $d/(2^d-1)$\\
$Q_2^+$ & $4$ & $3$ & $2/3$ & $3/4$\\
\bottomrule
\end{tabular}
\end{center}

\begin{proposition}\label{examples:values}
The values in the table are correct. Moreover, $\pi(C_h)=h-1$ and $\pi(V_s)=1$.
\end{proposition}

\begin{proof}
A copy of $C_h$ in $\mathcal W_{n,h}$ meets every level, so after removing the common intersection its type is a maximal chain of $2^{[h-1]}$, which is the only type. Any $j$ of its members have $|\bigcup\mathcal J\setminus\bigcap\mathcal J|\ge j-1$, with equality for $j$ consecutive members, so $s_j=j-1$ and the first row follows. Taking $h-1$ full levels gives $\pi(C_h)\ge h-1$. Conversely, a $C_h$-free family $\mathcal F$ has at most $h-1$ members on a uniformly random maximal chain of $2^{[n]}$, so $\sum_{A\in\mathcal F}\binom{n}{|A|}^{-1}\le h-1$ and $|\mathcal F|\le(h-1)M_n$.

In a copy of $V_s$ in two consecutive levels, every maximal member is the minimal member plus one coordinate, and these coordinates are distinct, so the only type is $\{\emptyset,\{1\},\dots,\{s\}\}$. For $j\ge2$, a subfamily containing the minimal member and $j-1$ maximal members has $|\bigcup\mathcal J\setminus\bigcap\mathcal J|=j-1$, and one consisting of maximal members only has $j$. This gives the second row. For $\pi(V_s)$, write the two levels as $k$ and $k+1$, let $\mathcal F$ be $V_s$-free, and let $a$ and $b$ be the numbers of members of $\mathcal F$ in the two levels. Every member of $\mathcal F$ in level $k$ has at most $s-1$ supersets in $\mathcal F$, so counting the subsets of members of $\mathcal F$ in level $k+1$ gives $(k+1)b\le(s-1)a+(n-k)(\binom nk-a)$. Since $k=n/2+O(1)$, this gives $a+b\le(1+O_s(n^{-1}))M_n$, and one full level shows that $\pi(V_s)=1$.

For the crown, let $A_1,A_2,A_3$ be the three members in level $k$; the three members in level $k+1$ are their pairwise unions, which must be distinct. Write $A_1=C\cup\{a\}$ and $A_2=C\cup\{b\}$. Since $|A_2\cup A_3|=|A_3\cup A_1|=k+1$, either $A_3=C\cup\{c\}$ with $c\notin\{a,b\}$, or $A_3\subseteq C\cup\{a,b\}$; in the second case the three unions coincide. Hence the only type is $\{\{1\},\{2\},\{3\},\{1,2\},\{1,3\},\{2,3\}\}$, with $(s_2,\dots,s_6)=(1,2,3,3,3)$: the values $s_2=1$ and $s_3=2$ are attained by a comparable pair and by $\{1\},\{2\},\{1,2\}$, and $s_4=3$ because two consecutive levels of an interval of dimension two contain at most three sets. This gives the third row.

For $Q_d$, Lemma~\ref{lem:2.1} shows that the only type is $2^{[d]}$, and $s_j=\lceil\log_2j\rceil$ by the proof of Lemma~\ref{lem:5.1}; Lemma~\ref{lem:5.2} gives $\nu(Q_d)=d/(2^d-1)$.

Finally, the four lower members of a copy of $Q_2^+$ lie in the first three levels and form a canonical diamond by Lemma~\ref{lem:2.1}, and the maximal member adds one coordinate. The only type is $\{\emptyset,\{1\},\{2\},\{1,2\},\{1,2,3\}\}$, with $(s_2,s_3,s_4,s_5)=(1,2,2,3)$, so $\nu(Q_2^+)=2/3<3/4=\kappa(Q_2^+)$.
\end{proof}

The crown is a non-Boolean example with matching exponents: in two middle levels its threshold order is $n^{-3/5}$, with coefficient $\pi(O_6)$ above the threshold and $2$ below it. The last row shows the gain from deleting subfamilies rather than whole copies in the sparse regime. Deleting one set from every diamond in the first three levels of $\mathcal W_{n,4}(p)$ leaves a $Q_2^+$-free family, and the expected number of such diamonds is $O(p^4n^2M_n)=o(pM_n)$ when $pn^{2/3}\to0$, so the threshold order for $Q_2^+$ is $n^{-2/3}$ although $\kappa(Q_2^+)=3/4$.

Some graded posets have no copy in their height window. Let $D_s$ be the poset with one minimal element, $s$ pairwise incomparable middle elements and one maximal element, so that $D_1=C_3$ and $D_2=Q_2$. For $s\ge3$ there is no weak copy of $D_s$ in three consecutive levels, since the minimal and maximal members differ in exactly two coordinates and so have only two sets strictly between them. Hence $\pi(D_s)=3$ and every subfamily of $\mathcal W_{n,3}$ is $D_s$-free. Similarly the butterfly, with two minimal elements below two maximal elements, has no weak copy in two consecutive levels: the union of two distinct sets in level $k$ has at least $k+1$ elements, so the two sets in level $k+1$ would coincide. Its window density is $2$. In these cases the enumeration and random statements are identities, and no exponent is defined.

The following tree poset has two types. Let $R$ be the poset with minimal elements $a,b,e$, maximal elements $c,d$ and relations $a<c$, $b<c$, $b<d$, $e<d$; its Hasse diagram is a path on five vertices. Let $B$ be the image of $b$ and $B\cup\{x\}$, $B\cup\{y\}$ the images of $c$ and $d$, where $x\ne y$. The images of $a$ and $e$ are $(B\cup\{x\})\setminus\{u\}$ and $(B\cup\{y\})\setminus\{v\}$ with $u,v\in B$, and they are distinct from $B$ and from each other. If $u=v$ the type is
\[
\mathcal T_3=\{\{1\},\{2\},\{3\},\{1,2\},\{2,3\}\},
\]
and if $u\ne v$ it is
\[
\mathcal T_4=\{\{1,2\},\{2,3\},\{1,4\},\{1,2,3\},\{1,2,4\}\},
\]
where the subscript is the dimension. No further comparabilities occur, since two distinct sets in level $k$ have at most one common superset in level $k+1$. The values of $s_j$ are
\[
\begin{array}{c|cccc|c}
\mathcal T&s_2&s_3&s_4&s_5&\nu_{\mathcal T}\\ \hline
\mathcal T_3&1&2&3&3&3/4\\
\mathcal T_4&1&2&3&4&1
\end{array}
\]
Here $s_2=1$ and $s_3=2$ are attained by a comparable pair and by two sets in level $k$ with their common superset, $s_4\ge3$ because an interval of dimension two contains at most three sets in two consecutive levels, and $s_4=3$ is attained for $\mathcal T_4$ by $\{1,2\},\{1,4\},\{1,2,3\},\{1,2,4\}$. Thus $x_-(R)=3$, $x_+(R)=4$, $\nu(R)=3/4$ and $\nu_+(R)=1$. Moreover $\pi(R)=1$: if $\mathcal F$ is $R$-free, then the inclusion graph between its members in levels $k$ and $k+1$ has no cycle of length four, and every longer cycle contains a path on five vertices, so this graph is a forest and its number $E$ of edges satisfies $E\le a+b$, where $a$ and $b$ count the members of $\mathcal F$ in the two levels; on the other hand $E\ge(n-k)a+(k+1)b-(n-k)\binom nk$ by counting the subsets of members in level $k+1$, and since $k=n/2+O(1)$ these inequalities give $a+b\le(1+O(n^{-1}))M_n$. Theorem~\ref{gen:connected-main} therefore gives $\alpha_R(\mathcal W_{n,2}(p))=(1+o(1))pM_n$ with high probability when $pn^{3/4}\to\infty$, and $\alpha_R(\mathcal W_{n,2}(p))=(2+o(1))pM_n$ when $pn\to0$ and $pM_n\to\infty$. For this poset the theorem of Jiang, Longbrake, Spiro and Yepremyan~\cite{Jiang2025Tree} is stronger: restricting their upper bound for the random Boolean lattice to two levels gives the dense conclusion as soon as $pn\to\infty$, so the threshold order of $R$ is $n^{-1}$, and the gap between $\nu(R)$ and $\nu_+(R)$ is a limitation of the method.

Finally we compare $\nu(P)$ with the parameter of Jiang, Longbrake and Yepremyan~\cite{Jiang2026Number}. For a finite poset $D$, let $d^*(D)$ be the smallest dimension of the interval spanned by an induced copy of $D$ in a Boolean lattice, and let
\[
\mu(P)=\min_{D\subseteq P,\ |D|\ge2}\frac{d^*(D)}{|D|-1}.
\]
Their copies may use arbitrary levels, while our types are confined to the height window. In the induced version of Theorem~\ref{gen:connected-main}, every subfamily $\mathcal J$ of a type is an induced copy of a subposet $D$ of $P$ with $|D|=|\mathcal J|$, so $|\bigcup\mathcal J\setminus\bigcap\mathcal J|\ge d^*(D)$ and hence $\nu_{\mathrm{ind}}(P)\ge\mu(P)$. This argument does not compare $\mu(P)$ with the weak parameter, since a weak embedding may add comparabilities. The inequality can be strict even when weak and induced copies in the window coincide: for $V_5$ we have $\nu(V_5)=\nu_{\mathrm{ind}}(V_5)=1$, whereas $\mu(V_5)=4/5$, because five pairwise incomparable sets can be chosen in the middle level of $2^{[4]}$ but not in $2^{[3]}$, so $d^*(V_5)=4$, and every other subposet of $V_5$ gives a ratio of at least one. For $C_h$, $Q_d$ and $O_6$ the two parameters agree, with values $1$, $d/(2^d-1)$ and $3/5$.

\section*{Declaration on the use of generative AI}

During the preparation of this paper, the authors used generative AI tools to polish the exposition, to check the arguments and to verify the references and numerical constants. All mathematical content was conceived, verified and written by the authors, who reviewed and edited every AI-assisted passage and take full responsibility for the content of the paper.

\end{document}